\documentclass[12 pt]{article}
\usepackage{stmaryrd}
\usepackage{amssymb}
\usepackage{amsmath}
\usepackage{amsthm}
\allowdisplaybreaks
\usepackage{amscd}
\usepackage{graphicx}
\usepackage{hyperref}
\usepackage{geometry}
\usepackage[all]{xy}

\usepackage{mathpazo}
\usepackage{avant}

\def\vbar{\mathchoice{\vrule height6.3ptdepth-.5ptwidth.8pt\kern- .8pt}
{\vrule height6.3ptdepth-.5ptwidth.8pt\kern-.8pt} {\vrule
height4.1ptdepth-.35ptwidth.6pt\kern-.6pt} {\vrule
height3.1ptdepth-.25ptwidth.5pt\kern-.5pt}}

\def\K{\mathbb{K}}
\def\<{\langle}
\def\>{\rangle}
\def\a{\alpha}
\def\b{\beta}

\def\c{\cdot}

\def\f{\phi}
\def\p{\psi}
\def\r{\rho}

\newtheorem{thm}{Theorem}[section]
\newtheorem{lem}[thm]{Lemma}
\newtheorem{cor}[thm]{Corollary}
\newtheorem{pro}[thm]{Proposition}
\newtheorem{ex}[thm]{Example}

\theoremstyle{definition}
\newtheorem{df}{Definition}[section]

\theoremstyle{remark}
\newtheorem{rmk}{Remark}[section]
\begin{document}
\title{Constructions and representation theory of  BiHom-post-Lie algebras}
\author{\bf H. Adimi, T. Chtioui, S. Mabrouk, S. Massoud}
\author{H. Adimi$^{1}$
\footnote{Corresponding author,  E-mail: h.adimi@univ-bba.dz},
T. Chtioui$^{2}$
\footnote{Corresponding author,  E-mail: chtioui.taoufik@yahoo.fr}
, S. Mabrouk$^{3}$
\footnote{Corresponding author,  E-mail: Mabrouksami00@yahoo.fr}
, S. Massoud$^{4}$
\footnote{Corresponding author,  E-mail: sonia.massoud2015@gmail.com}
\\
{\small 1.   Universit\'e Mohamed El Bachir El Ibrahimi de Bordj Bou Arr\'eridj El-Anasser, 34030 - Alg\'erie, } \\
{\small 2.  Faculty of Sciences, University of Sfax, BP 1171, 3000 Sfax, Tunisia}\\
{\small 3.  University of Gafsa, Faculty of Sciences Gafsa, 2112 Gafsa, Tunisia}\\
{\small 4.  Faculty of Sciences, University of Sfax, BP 1171, 3000 Sfax, Tunisia}
}
\date{}
\maketitle
\begin{abstract}
The main goal of this paper is to give some construction results of BiHom-post-Lie algebras which are a generalization of both post-Lie-algebras and Hom-post-Lie
algebras. They are the algebraic structures behind the $\mathcal{O}$-operator of BiHom-Lie algebras. They can be also regarded as the splitting into three parts of the structure
of a BiHom-Lie-algebra. Moreover we develop the representation theory of BiHom-post-Lie algebras on a vector space $V$. We show that there is naturally an induced
representation of its sub-adjacent Lie algebra.
\end{abstract}

\textbf{Key words}: BiHom-post-Lie algebras, BiHom-Lie algebras, representations,   $\mathcal{O}$-operator.

\textbf{Mathematics Subject Classification}: 17D05,17D10,17A15.
\section{Introduction}

The notion of post-algebras goes back to Rosenbloom (1942). An equivalent
formulation of the class of post-algebras was given by Rousseau (1969, 1970)
which became a starting point for deep research. After then authors inspired
many applications in computer science of various generalizations of
post-algebras.\newline
Post-Lie algebras have been introduced by Vallette in 2007 \cite{B. Vallette}
in connection with the homology of partition posets and the study of Koszul
operads. However, J. L. Loday studied pre-Lie algebras and post-Lie algebras
within the context of algebraic operad triples, see for more details \cite%
{Dialg99,Trialg01}. In the last decade, many works \cite{Burde16,Fard,sl2}
intrested in post-Lie algebra structures, motivated by the importance of
pre-Lie algebras in geometry and in connection with generalized Lie algebra
derivations.\newline
Recently,Post-Lie algebras which are non-associative algebras played an
important role in different areas of pure and applied mathematics. They
consist of a vector space $A$ equipped with a Lie bracket $[\cdot ,\cdot ]$
and a binary operation $\triangleright $ satisfying the following axioms%
\begin{align}
x\triangleright \lbrack y,z]& =[x\triangleright y,z]+[y,x\triangleright z],
\\
\lbrack x,y]\triangleright z& =as_{\triangleright
}(x,y,z)-as_{\triangleright }(y,x,z).
\end{align}%
If the bracket $[\cdot ,\cdot ]$ is zero, we have exactly a pre-Lie
structure. The varieties of pre- and post-Lie algebras play a crucial role
in the definition of any pre and post-algebra through black Manin operads
product, see details in \cite{BBGN2012,GubKol2014}. Whereas pre-Lie algebras
are intimately associated with euclidean geometry, post-Lie algebras occur
naturally in the differential geometry of homogeneous spaces, and are also
closely related to Cartan's method of moving frames .Fard et al.\cite{Fard}
studied universal enveloping algebras of post-Lie algebras and the free
post-Lie algebra.

In \cite{bakayoko} Bakayoko studied hom-post-Lie algebras, as a twisted
generalization post Lie algebra. He studied modules over Hom-post-Lie
algebras and gave some constructions he showed also that modules over
Hom-post-Lie algebras are close by twisting either by Hom-post-Lie algebra
endomorphisms or module structure maps. We can obtain Hom post-Lie modules
from only a given multiplicative Hom-post-Lie algebra in a non-trivial sense.%
\newline
Motivated by a categorical study of Hom-algebra and new type of categories,
the authors of \cite{bihomass} introduced a generalized algebraic structure
dealing with two commuting multiplicative linear maps, called BiHom-algebras
including BiHom-associative algebras and BiHom-Lie algebras. When the two
linear maps are the same, then BiHom-algebras become Hom-algebras in some
cases. Many varieties of algebras are generalized to the BiHom-version, see
details in \cite{bihomalt,bihomprealt,Liu&Makhlouf&Menini&Panaite}. \newline

The paper is organized as follows. In section 2, we recall some basic
definitions and properties of post-Lie algebras and BiHom-Lie algebras. In
section 3,  We construct BiHom-Lie algebras structure associated to any
BiHom-post-Lie algebra. The bracket is given by%
\begin{equation*}
\left\{ x,y\right\} =x\triangleright y-\alpha ^{-1}\beta \left( y\right)
\triangleright \alpha \beta ^{-1}\left( x\right) +\left[ x,y\right] .
\end{equation*}%
In addition, we investigate the notion of an $\mathcal{O}$-operator of
weight $\lambda $ to construct a BiHom-post-Lie algebra structure on the $A$%
-module $\mathbb{K}$-algebra of a BiHom-Lie algebra $(A,[\cdot ,\cdot
],\alpha ,\beta )$. The Last section is devoted to introduce the
representation theory of a BiHom-post-Lie algera and its sub-adjacent
BiHom-Lie algebra. \newline
Throughout this paper, all vector spaces are over a field $\mathbb{K}$ of
characteristic zero.
\section{Preliminaries}
In what follows, we recall some concepts and facts used in this paper
\begin{df}\label{post-lie alg}
A (left) post-Lie algebra $(A,[\c,\c],\rhd)$ consists of a Lie algebra $(A,[\c,\c])$
and a binary product $\rhd:A \times A \to A$ such that, for all elements $x, y, z  \in A$ the following relations hold
\begin{align}
    & x \rhd [y,z]=[x\rhd y,z]+[y,x \rhd z], \label{postlie1} \\
    &[x,y]\rhd z= as_{\rhd}(x,y,z)-as_{\rhd}(y,x,z), \label{postlie2}
\end{align}
where $as_{\rhd}(x,y,z)=x\rhd (y \rhd z)-(x\rhd y)\rhd z$.
\end{df}
It is clear that a post-Lie algebra with an abelian Lie algebra structure reduces to a pre-Lie algebra.
If we define $L_{\rhd}: A \to gl(A)$ by $L_{\rhd}(x)y=x\rhd y$, then by Eq. \eqref{postlie1},   $L_{\rhd}(x)$ is a derivation on $(A,[\c,\c])$.

A morphism $f:(A,[\c,\c],\rhd)\to (A',[\c,\c]',\rhd')$ of post-Lie algebras is a linear map satisfying
 $$f([x,y])=[f(x),f(y)], \ \ f(x\rhd y)=f(x)\rhd f(y), \ \ \forall x,y\in A.$$

\begin{pro}
Let  $(A,[\c,\c],\rhd)$ be a post-Lie algebra. Then the bracket
\begin{align}\label{post-lie =>lie}
    \{x,y\}=x\rhd y-y\rhd x+[x,y]
\end{align}
defines a Lie algebra structure on $A$.
\end{pro}
We denote this algebra by $A^C$ and it is called the sub-adjacent Lie-algebra of  $(A,[\c,\c],\rhd)$.
\begin{df}\label{RepPostLie}
 A representation of a   post-Lie algebra $(A,[\c,\c],\rhd)$ on a vector space $V$ is a tuple $(V,\rho,\mu,\nu)$, such that $(V,\rho)$ is a representation of
$(A,[\c,\c])$  and $\mu,\nu: A \to gl(V)$ are linear maps satisfying
{\small
\begin{align}
& \nu([x,y])  =\rho(x) \nu(y) -\rho(y) \nu(x)  , \label{rep pos lie 1}\\
& \rho(x\rhd y) =\mu(x) \rho(y)-\rho(y)  \mu(x)  , \label{rep pos lie 2}\\
&\mu([x,y]) =\mu(x)\mu(y)-\mu(x\rhd y)  -\mu(y)\mu(x)+\mu(y\rhd x)     , \label{rep pos lie 3}\\
& \nu(y)\rho(x) =\mu(x)\nu(y)-\nu(y)\mu(x)-\nu(x\rhd y) +\nu(y)\nu(x)  ,  \label{rep pos lie 4}
\end{align}}
for any $x,y \in A$.
 \end{df}

Let  $(A,[\c,\c],\rhd)$ be a post-Lie algebra and  $(V,\rho,\mu,\nu)$  a representation of   $(A,[\c,\c],\rhd)$.  By identity \eqref{rep pos lie 3}, we deduce that $(V,\mu)$ is a representation of the sub-adjacent Lie algebra $(A,\{\c,\c\})$ of   $(A,[\c,\c],\rhd)$.  Further, it is obvious that $(A,ad,L_{\rhd},R_{\rhd})$ is a representation of  $(A,[\c,\c],\rhd)$  which is called the adjoint representation.

\begin{pro}Let   $(V,\rho,\mu,\nu)$ be a representation of a post-Lie algebra   $(A,[\c,\c],\rhd)$.  Then $(V,\r+\mu-\nu)$ is  a representation of the sub-adjacent Lie algebra $(A,\{\c,\c\})$ of   $(A,[\c,\c],\rhd)$.
\end{pro}
\begin{df}
A BiHom-Lie algebra is a tuple  $(A,[\cdot,\cdot],\a,\b)$ consisting of a linear space $A$, a linear map $[\cdot,\cdot]:\otimes^2 A\longrightarrow A$ and two linear map $\a,\b:A\longrightarrow A $, satisfying
\begin{eqnarray}
 &\a \circ \b =\b \circ \a, \nonumber\\
&\a( [x,y]) =[\a(x),\a(y)],\ \b( [x,y]) = [\b(x),\b(y)],  \nonumber\\
& [\b(x),\a(y)]= -[\b(y),\a(x)],\nonumber\\
& \circlearrowleft_{x,y,z \in A} [\b^2(x),[\b(y),\a(z)]]= 0\label{Bihom-jaco},
\end{eqnarray}
for any $x,y,z \in A$. A BiHom-Lie algebra  $(A,[\cdot,\cdot],\a,\b)$ is called regular if $\a$ and $\b$ are bijective.
\end{df}
\begin{pro}\label{TwistLieAlg}
   Let $(A,[\cdot,\cdot])$ be a Lie algebra and $\a,\b$ two commuting morphisms on $A$. Then $(A,[\cdot,\cdot]_{\a,\b}=[\cdot,\cdot]\circ (\a\otimes\b),\alpha,\beta)$ is a BiHom-Lie algebra.
\end{pro}

 \begin{df}\label{defi:bihom-lie representation}
 A representation of a BiHom-Lie algebra $(A,[\cdot,\cdot],\alpha,\beta)$ on
 a vector space $V$ with respect to commuting linear maps $\phi,\psi:V\rightarrow V$ is a linear map
  $\rho:A\longrightarrow gl(V)$, such that for all
  $x,y\in A$, the following equalities are satisfied
\begin{eqnarray}
\label{bihom-lie-rep-1}\rho(\alpha(x))\circ \f&=&\f\circ \rho(x),\\
\label{bihom-lie-rep-2} \rho(\b(x))\circ \p&=&\p \circ \rho(x),\\
\label{bihom-lie-rep-3}\rho([\b(x),y])\circ  \p &=&\rho(\alpha \b(x))\circ\rho(y)-\rho(\b(y))\circ\rho(\a(x)).
\end{eqnarray}
  \end{df}
We denote such a representation by $(V,\rho,\f,\p)$. For all $x\in A$, we define $ad_{x}:A \to  A$ by
\begin{eqnarray*}
ad_{x}(y)=[x,y],\quad\forall y \in A.
\end{eqnarray*}
Then $ad:A \longrightarrow gl(A)$ is a representation of the BiHom-Lie algebra $(A,[\cdot,\cdot],\a,\b)$ on $A$ with respect to $\alpha$ and $\b$,  which is called the adjoint representation.

  Let $(A,[\cdot,\cdot])$ be a Lie algebra, $\a,\b$ two commuting morphisms on $A$. Consider a representation $(V,\rho)$ of $A$ and two commuting linear maps $\phi,\psi:V\rightarrow V$  satisfying
  $$\rho(\alpha(x))\circ \f=\f\circ \rho(x),\ \
 \rho(\b(x))\circ \p=\p \circ \rho(x).$$
  Define a linear map   $\widetilde{\rho}:A\longrightarrow gl(V)$ by $\widetilde{\rho}(x)(v)=\rho(\alpha(x))( \p(v))$.\begin{lem}\label{twistmodule} With the above notation $(V,\widetilde{\rho},\f,\p)$ is a representation of $(A,[\cdot,\cdot]_{\a,\b},\alpha,\beta)$.
\end{lem}
\begin{lem}\label{lem:semidirectp}
Let $(A,[\cdot,\cdot],\alpha,\b)$ be a regular BiHom-Lie algebra, $(V,\f,\p)$  a vector space with two commuting bijective  linear transformations and $\rho:
A \rightarrow gl(V)$ a linear
map. Then $(V,\r,\f,\p)$ is a representation of $(A,[\cdot,\cdot],\a,\b)$ if and only if $(A\oplus V,[\cdot,\cdot]_\rho,\alpha+\f,\b+\p)$ is a BiHom-Lie algebra, where $[\cdot,\cdot]_\rho$, $\alpha+\f$ and $\b+\p$ are defined by
\begin{eqnarray}\label{eq:sum}
[x+u,y+v]_{\rho}&=&[x,y]+\rho(x)v-\rho(\a^{-1}\b(y))\f\p^{-1}u,\\
(\alpha+\f)(x+u)&=&\alpha(x)+\f(u),\\
(\b+\p)(x+u)&=&\b(x)+\p(u),
\end{eqnarray}
for all $x,y\in A,~u,v\in V$.
\end{lem}

Now we recall the definition  of an $\mathcal{O}$-operator on a BiHom-Lie algebra associated to a given representation, which generalize the Rota-Baxter operator of weight $0$ introduced in \cite{Liu&Makhlouf&Menini&Panaite} and defined as a linear operator $R$ on a BiHom-Lie algebra  $(A,[\cdot,\cdot],\alpha,\beta)$ such that $R\circ \a=\a \circ R$, $R \circ \b=\b \circ R$ and
$$[R(x),R(y)]=R([R(x),y]+[x,R(y)]),\ \forall x,y \in A.$$
\begin{df}
Let $(A,[\cdot,\cdot],\alpha,\beta)$ be a BiHom-Lie algebra and $(V,\rho,\f,\p)$ be a representation.  A linear map $T: V \to A$ is called an $\mathcal{O}$-\textbf{operator}
associated to $\rho$, if it satisfies
\begin{align}\label{O-operator}
    [T(u),T(v)]=T(\rho(T(u))v-\rho(T(\f^{-1}\p(v)))\f\p^{-1}(u)),\ \ \forall u,v \in V.
\end{align}
\end{df}

\section{BiHom-post-Lie algebras}

In this section, we recall  the notion of BiHom-post-Lie algebras  ( see \cite{Liu&Makhlouf&Menini&Panaite}). We will provide some construction results.

\begin{df}
Let $(A,[\c,\c],\a,\b)$ be BiHom-Lie algebra. Let $(V,\{\c,\c\},\f,\p)$ be a BiHom-Lie algebra and $\rho:A  \to gl(V)$ be a linear map. We say that $(V,\{\c,\c\},\rho,\f,\p,)$ is an $A$-module $\mathbb{K}$-algebra if $(V,\rho,\f,\p,)$ is a representation of  $(A,[\c,\c],\a,\b)$  such that the following compatibility condition holds
\begin{align}\label{repKalgebras}
&\r(\a\b(x))\{u,v\}=\{\r(\b(x))u,\p(v)\}+\{\p(u),\r(\a(x))v\},
\end{align}
for all $x \in A,\ u,v \in V$.
\end{df}
It is known  that $(A,ad,\a,\b)$ is a representation of $A$ called the adjoint representation. Then $(A,[\cdot,\cdot],ad,\a,\b)$ is an $A$-module $\mathbb{K}$-algebra.
\begin{pro}
Let $(A,[\c,\c],\a,\b)$ and $(V,\{\c,\c\},\f,\p)$ be two regular BiHom-Lie algebras and $\r: A \to gl(V)$ be a linear map. Then $(V,\{\c,\c\},\r,\f,\p)$  be an $A$-module $\mathbb{K}$-algebra if and only if the direct sum $A\oplus V$ of vector spaces  is turned into a  BiHom-Lie algebra (the semi-direct sum) by defining a bracket on  $A\oplus V$ by
\begin{align}
[x+u,y+v]_{\rho}=&[x,y]+\rho(x)v-\rho(\a^{-1}\b(y))\f\p^{-1}u+\{u,v\},\\
(\alpha+\f)(x+u)=&\alpha(x)+\f(u),\\
(\b+\p)(x+u)=&\b(x)+\p(u),
\end{align}
for all $x,y\in A,~u,v\in V$.
\end{pro}
We denote this algebra by $A\ltimes_{\r,\f,\p}^{\a,\b} V$ or simply $A\ltimes V$.
\begin{proof}Let $x,y,z\in A$ and $a,b,c\in V$
  \begin{align*}
&\circlearrowleft_{x+a,y+b,z+c}[(\beta+\psi)^2(x+a),[(\beta+\psi)(y+b),(\alpha+\phi)(z+c)]_{\rho}]_{\rho}\\
=&\circlearrowleft_{x+a,y+b,z+c}([\beta^2(x),[\beta(y),\alpha(z)]]+\rho(\beta^2(x))\rho(\beta(y))\phi(c)-\rho(\beta^2(x))\rho(\beta(z))\phi(b)\\
&+\rho(\beta^2(x))[\psi(b),\phi(c)]_{V}-\rho([\alpha^{-1}\beta^2(y),\beta(z)])\phi\psi(a)+[\psi^2(a),\rho(\beta(y))\phi(c)]_{V}\\
&-[\psi^2(a),\rho(\beta(z))\phi(b)]_{V}+[\psi^2(a),[\psi(b),\phi(c)]_{V}]_{V})=0,
\end{align*}
 if and only if $(V,\rho,\phi,\psi)$ is a representation on $A$ and satisfies Eq.  \eqref{repKalgebras}.
\end{proof}
\begin{df}
 Let $(A,[\c,\c],\a,\b)$ be a regular BiHom-Lie algebra and  $(V,\{\c,\c\},\r,\f,\p)$  be a regular $A$-module $\mathbb{K}$-algebra. A linear map $T:V\to A$ is called an $\mathcal{O}$-operator of weight $\lambda\in \mathbb{K}$ associated to $(V,\{\c,\c\},\r,\f,\p)$ if its satisfies
 \begin{align}
    & T\f=\a T,\ \  T\p=\b T, \\
    & \label{O-operatorLamda}
    [T(u),T(v)]=T(\rho(T(u))v-\rho(T(\f^{-1}\p(v)))\f\p^{-1}(u)+\lambda\{u,v\}),
 \end{align}
 for all $u,v\in V$.
\end{df}
\begin{ex}
  An $\mathcal{O}$-operator of weight $\lambda$ associated to $(A,[\c,\c],ad,\a,\b)$ is just  a Rota-Baxter on $A$ of the same weight.
\end{ex}
\begin{df}\cite{Liu&Makhlouf&Menini&Panaite}
A (left) BiHom-post-Lie algebra  is a tuple $(A,[\c,\c],\rhd,\a,\b)$ where $(A,[\c,\c],\a,\b)$ is a BiHom-Lie algebra such that for any $x,y,z \in A$
\begin{align}
& \a(x \rhd y)=\a(x) \rhd \a(y),\ \b(x \rhd y)= \b(x) \rhd \b(y), \\
\label{bihom-post-lie condition1}
    &\a\b(x)\rhd[y,z]=[\b(x)\rhd y,\b(z)]+[\b(y),\a(x)\rhd z] ,\\
    \label{bihom-post-lie condition2}& [\b(x),\a(y)]\rhd \b(z)=as_{\a,\b}(\b(x),\a(y),z)- as_{\a,\b}(\b(y),\a(x),z) ,
\end{align}
where $as_{\a,\b}(x,y,z)=\a(x)\rhd (y\rhd z)-(x\rhd y) \rhd \b(z)$.  We call $[\c,\c]$ the \emph{torsion}  and $\rhd$ the  \emph{connection}  of the (left) BiHom-post-Lie algebra.
\end{df}
if $\a=\b=id$, then we recover a (left) post-Lie algebra. In addition, if $[\c,\c]=0$, then $A$ reduces to a  BiHom-pre-Lie algebra.

From now on, we will write BiHom-post-Lie algebra instead of (left) BiHom-post-Lie algebra.
A BiHom-post-Lie algebra  $(A,[\c,\c],\rhd,\a,\b)$ is said to be regular if $\a$ and $\b$ are bijective.


\begin{ex}\label{TwistBiHomPostLie}
Let $(A,[\c,\c],\rhd)$   be a post-Lie algebra and $\a, \b$ be two commuting morphisms on $A$. Then $(A,[\c,\c]_{\a,\b},\rhd_{\a,\b})$ is a BiHom-post-Lie algebra, Where, for $x,y\in A$,

\begin{equation}\label{TwistOperations}
  [x,y]_{\a,\b}=[\a(x),\b(y)],\ \ x\rhd_{\a,\b}y=\a(x)\rhd\b(y).
\end{equation}
\end{ex}

\begin{ex}
  Let $(A,[\c,\c],\a,\b)$ be a  BiHom-Lie algebra. Then  $(A,[\c,\c],\rhd,\a,\b)$ is a BiHom-post-Lie algebra, where
  $$x\rhd y=[y,x],\ \ \forall x,y\in A.$$
\end{ex}
\begin{pro}\label{subadj lie alg}
Let $(A,[\c,\c],\rhd,\a,\b)$ be a regular BiHom-post-Lie algebra. Then the bracket
\begin{equation}\label{liebracket}
      \{x,y\}=x\rhd y-\a^{-1}\b(y)\rhd \a\b^{-1}(x)+[x,y]
\end{equation}
defines a BiHom-Lie algebra structure on $A$. We denote this algebra by $A^C$ and we call it the sub-adjacent  BiHom-Lie algebra of $A$.
\end{pro}
\begin{proof}
The BiHom-skew symmetry  is obvious. We will just check the BiHom-Jacobi identity. For all $x,y,z \in A,$ we have
\begin{align*}
\{\beta^2(x),\{\beta(y),\alpha(z)\}\}&=\beta^2(x)\rhd(\beta(y)\rhd\alpha(z))-(\alpha^{-1}\beta^2(y))\rhd\alpha\beta(z))\rhd\alpha\beta(x)\\
&+[\beta^2(x),\beta(y)\rhd\alpha(z)]-\beta^2(x)\rhd(\beta(z)\rhd\alpha(y))\\
&+(\alpha^{-1}\beta^2(z)\rhd\beta(y))\rhd\alpha\beta(x)-[\beta^2(x),\beta(z)\rhd\alpha(y)]\\
&+\beta^2(x)\rhd[\beta(y),\alpha(z)]+[\alpha^{-1}\beta^2(y),\beta(z)]\rhd\alpha\beta(x)\\
&+[\beta^2(x),[\beta(y),\alpha(z)]],
\end{align*}
\begin{align*}
\{\beta^2(y),\{\beta(z),\alpha(x)\}\}&=\beta^2(y)\rhd(\beta(z)\rhd\alpha(x))-(\alpha^{-1}\beta^2(z))\rhd\beta(x))\rhd\alpha\beta(y)\\
&+[\beta^2(y),\beta(z)\rhd\alpha(x)]-\beta^2(y)\rhd(\beta(x)\rhd\alpha(z))\\
&+(\alpha^{-1}\beta^2(x)\rhd\beta(z))\rhd\alpha\beta(y)-[\beta^2(y),\beta(x)\rhd\alpha(z)]\\
&+\beta^2(y)\rhd[\beta(z),\alpha(x)]+[\alpha^{-1}\beta^2(z),\beta(x)]\rhd\alpha\beta(y)\\
&+[\beta^2(y),[\beta(z),\alpha(x)]].
\end{align*}
Similarly,
\begin{align*}
\{\beta^2(z),\{\beta(x),\alpha(y)\}\}&=\beta^2(z)\rhd(\beta(x)\rhd\alpha(y))-(\alpha^{-1}\beta^2(x))\rhd\beta(y))\rhd\alpha\beta(z)\\
&+[\beta^2(z),\beta(zx)\rhd\alpha(y)]-\beta^2(z)\rhd(\beta(y)\rhd\alpha(x))\\
&+(\alpha^{-1}\beta^2(y)\rhd\beta(x))\rhd\alpha\beta(z)-[\beta^2(z),\beta(y)\rhd\alpha(x)]\\
&+\beta^2(z)\rhd[\beta(z),\alpha(y)]+[\alpha^{-1}\beta^2(x),\beta(y)]\rhd\alpha\beta(z)\\
&+[\beta^2(z),[\beta(x),\alpha(y)]].
\end{align*}

By the BiHom-Jacobi identity of BiHom-Lie algebras and  \eqref{bihom-post-lie condition1} and \eqref{bihom-post-lie condition2} we have
$$\{\beta^2(x),\{\beta(y),\alpha(z)\}\}+\{\beta^2(y),\{\beta(z),\alpha(x)\}\}+\{\beta^2(z),\{\beta(x),\alpha(y)\}\}=0.$$
\end{proof}

\begin{rmk}
Given a BiHom-post-Lie algebra  $(A,[\c,\c],\rhd,\a,\b)$. Suppose that $\rhd$ is commutative in the BiHom-sense, that is
$\b(x) \rhd \a(y)=\b(y) \rhd \a(x)$. Then the two Lie brackets   $[\c,\c]$ and $\{\c,\c\}$ coincide.
\end{rmk}

Now, recall that a  regular BiHom-algebra $(A,\c,\a,\b)$ is called admissible BiHom-Lie algebra if $(A,[\c,\c],\a,\b)$ is a BiHom-Lie algebra, where
$[x,y]=x \c y- \a^{-1}\b(y) \c \a \b^{-1}(x).$
\begin{cor}
  Let $(A,[\c,\c],\rhd,\a,\b)$ be a regular BiHom-post-Lie algebra. Define the product $\circ$ as
  \begin{equation}\label{AdmissibileBiHomLie}
    x\circ y=x\vartriangleright y+\frac{1}{2}[x,y],\ \forall\ x,y\in A.
  \end{equation}
  Then $(A,\circ,\a,\b)$ is an admissible BiHom-Lie algebra.
\end{cor}

\begin{pro}
Let $(A,[\c,\c],\rhd,\a,\b)$ be a regular BiHom-post-Lie algebra. Then
$(A,L_\rhd,\a,\b)$ is a representation of  $(A,\{\c,\c\},\a,\b)$,
where the bracket $\{\c,\c\}$ is defined by the identity \eqref{liebracket}.
\end{pro}

\begin{proof}
It is obvious to check that
$L_\rhd(\a(x)) \circ \a= \a \circ L_\rhd(x) $ and $L_\rhd(\b(x)) \circ \b= \b \circ L_\rhd(x)$.
To prove \eqref{bihom-lie-rep-3},  we compute as follows.  Let $x,y,z \in A$. Then
\begin{align*}
&L_\rhd\Big( \{\b(x),y\}\Big)\b(z)= \{\b(x),y\} \rhd \b(z)\\
=& (\b(x)\rhd y)\rhd \b(z)-(\a^{-1}\b(y)\rhd \a\b^{-1}(x))\rhd \b(z)  + [\b(x),y] \rhd \b(z) \\
 =& (\b(x)\rhd y)\rhd \b(z)-(\a^{-1}\b(y)\rhd \a(x))\rhd \b(z) + \a\b(x)\rhd (y \rhd z) \\
 &-(\b(x) \rhd y) \rhd \b(z) -\b(y) \rhd (\a(x) \rhd z) + (\a^{-1}\b(y) \rhd \a(x)) \rhd \b(z) \\
  =& \a\b(x)\rhd (y \rhd z) -\b(y) \rhd (\a(x) \rhd z) \\
 =& L_\rhd(\a\b(x)) \circ L_\rhd(y) (z)- L_\rhd(\b(y)) \circ L_\rhd(\a(x)) (z)
\end{align*}
and the proof is finished.
\end{proof}

\begin{pro}
  Let $(A,[\c,\c],\rhd,\a,\b)$ be a BiHom-post-Lie algebra. Then  $(A,-[\c,\c],\blacktriangleright,\a,\b)$, where \begin{equation}\label{BlackBiHomPosteLie}
x\blacktriangleright y=x\vartriangleright y+[x,y],\end{equation} is  also a BiHom-post-Lie algebra. In addition $(A,[\c,\c],\rhd,\a,\b)$  and  $(A,-[\c,\c],\blacktriangleright,\a,\b)$ have the same sub-adjacent BiHom-Lie algebra $A^C$.
\end{pro}
\begin{proof}
For all $x,y,z\in A,$ we check that $(A,-[\cdot,\cdot],\blacktriangleright,\alpha,\beta)$ is BiHom-post-Lie algebras. in fact we have
\begin{align*}
-\alpha\beta(x)\blacktriangleright[y,z]&=-(\alpha\beta(x)\rhd[y,z]+[\alpha\beta(x),[y,z]])\\
&=-\alpha\beta(x)\rhd[y,z]-[\beta^2(\alpha\beta^{-1}(x)),[\beta(\beta^{-1}(y)),\alpha(\alpha^{-1}(z))]],
\end{align*}
\begin{align*}
-[\beta(x)\blacktriangleright y,\beta(z)]=&-[\beta(x)\rhd y+[\beta(x),y],\beta(z)]\\
=&-[\beta(x)\rhd y,\beta(z)]+[\alpha^{-1}\beta^2(\beta^{-1}(z)),[\alpha(x),\alpha\beta^{-1}(y))]]
\end{align*}
and
\begin{align*}
-[\beta(y),\alpha(x)\blacktriangleright z]=&-[\beta(y),\alpha(x)\rhd z+[\alpha(x),z]]\\
&=-[\beta(y),\alpha(x)\rhd z]+[\beta^2(\beta^{-1}(y)),[\beta\alpha^{-1}(z),\alpha(\alpha\beta^{-1}(x))]].
\end{align*}
Then we  obtain
\begin{align*}
&-\alpha\beta(x)\blacktriangleright[y,z]+[\beta(x)\blacktriangleright y,\beta(z)]+[\beta(y),\alpha(x)\blacktriangleright z]\\
&=-\alpha\beta(x)\rhd[y,z]-[\beta(x)\rhd y,\beta(z)]-[\beta(y),\alpha(x)\rhd z]\\
&-[\beta^2(\alpha\beta^{-1}(x)),[\beta(\beta^{-1}(y)),\alpha(\alpha^{-1}(z))]]\\
&+[\alpha^{-1}\beta^2(\beta^{-1}(z)),[\alpha(x),\alpha\beta^{-1}(y))]]\\
&+[\beta^2(\beta^{-1}(y)),[\beta\alpha^{-1}(z),\alpha(\alpha\beta^{-1}(x))]]=0.
\end{align*}
Hence,  the first condition hold using \eqref{bihom-post-lie condition1} and BiHom-Jacobi-Identity. To check the second condition, we compute as follow
\begin{align*}
&-[\beta(x),\alpha(y)]\blacktriangleright\beta(z)-as_{\alpha,\beta}^{\blacktriangleright}(\beta(x),\alpha(y),z)+as_{\alpha,\beta}^{\blacktriangleright}(\beta(y),\alpha(x),z)\\
&=-[\beta(x),\alpha(y)]\blacktriangleright\beta(z)+[[\beta(x),\alpha(y)],\beta(z)]+\alpha\beta(x)\rhd(\alpha(y)\rhd z)\\
&-(\beta(x)\rhd\alpha(y))\rhd \beta(z)+\alpha\beta(x)\rhd[\alpha(y),z]+[\alpha\beta(x),\alpha(y)\rhd z]\\
&+[\alpha\beta(x),[\alpha(y),z]]-[\beta(x),\alpha(y)]\rhd\beta(z)-[\beta(x)\rhd\alpha(y),\beta(z)]\\
&-[[\beta(x),\alpha(y)],\beta(z)]-\alpha\beta(y)\rhd(\alpha(x)\rhd z)+(\beta(y)\rhd \alpha(x))\rhd\beta(z)\\
&-\alpha\beta(y)\rhd[\alpha(x),z]-[\alpha\beta(y),\alpha(x)\rhd z]-[\alpha\beta(y),[\alpha(x),z]]\\
&+[\beta(y),\alpha(x)]\rhd\beta(z)+[\beta(y)\rhd \alpha(x),\beta(z)]+[[\beta(y),\alpha(x)],\beta(z)]\\
&=0,
\end{align*}
where $as_{\alpha,\beta}^{\blacktriangleright}(x,y,z)=x\blacktriangleright (y\blacktriangleright z)-(x\blacktriangleright y)\blacktriangleright z$.
\end{proof}
\begin{thm}
  Let $(A,[\c,\c],\rhd,\a,\b)$ be a BiHom-post-Lie algebra. Define the double bracket  $\llbracket \c,\c \rrbracket$ on $A\times A$ by
  \begin{align}
    \llbracket (a,x),(b,y) \rrbracket&=(a\rhd b-\a^{-1}\b(b)\rhd\a\b^{-1}(a)+[a,b],\nonumber\\&
   \label{DoubleBarcket} \qquad\qquad\qquad\qquad a\rhd y-\a^{-1}\b(b)\rhd\a\b^{-1}(x)+[x,y])
  \end{align}
  for all $a,b,x,y\in A.$ Then $(A\times A,\llbracket \c,\c \rrbracket,\a^{\times2},\b^{\times2})$ is  a BiHom-Lie algebra.
\end{thm}
\begin{proof}

Let $x,y,z,a,b,c\in A$, It’s obvious that $$\llbracket (\b(a),\b(x)),(\a(b),\a(y)) \rrbracket=-\llbracket (\b(b),\b(y)),(\a(a),\a(x)) \rrbracket.$$
 On the other hand, we have
\begin{align*}
 &\circlearrowleft_{(a,x),(b,y),(c,z)} \llbracket (\b^2(a),\b^2(x)),\llbracket (\b(b),\b(y)),(\a(c),\a(z)) \rrbracket \rrbracket
 \\=&\big(\circlearrowleft_{a,b,c} \{\b^2(a),\{\b(b),\a(c)\}\},\\&\circlearrowleft_{(a,x),(b,y),(c,z)}\b^2(a)\rhd (\b(b)\rhd \a(z))-\b^2(a)\rhd( \b(c)\rhd\a(y))+\b^2(a)\rhd [\b(y),\a(z)]
\\&\qquad\qquad\qquad+(\b(b)\rhd \a(c))\rhd \a\b(x)-(\b(c)\rhd\a(b))\rhd \a\b(x)
+[\b(b),\a(c)]\rhd \a\b(x)
\\&\qquad\qquad\qquad+[\b^2(x),\b(b)\rhd \a(z)]-[\b^2(x),\b(c)\rhd\a(y)]+[\b^2(x),[\b(y),\a(z)]]\big)
  \\ =&(0,0),
\end{align*}
where the bracket $\{\c,\c\}$ is defined in \eqref{liebracket}. Therefor $(A\times A,\llbracket \c,\c \rrbracket,\a^{\times2},\b^{\times2})$ is  a BiHom-Lie algebra.
\end{proof}
It is well known that BiHom-assciative, BiHom-pre-Lie and BiHom-Novikov algebras are BiHom-Lie admissible algebras. Another class of  BiHom-Lie admissible algebras is the variety of BiHom-LR algebras. A BiHom-LR algebra is a tuple $(A,  \cdot , \alpha , \beta )$ satisfying the following conditins
\begin{align}\label{LRCondition1}
   &(x\c y)\c\a(z)=(x\c z)\c\a(y) ,\\
  \label{LRCondition2} & \b(x)\c(y\c z)=\b(y)\c(x\c z),\ \forall x,y,z\in A.
\end{align}
\begin{pro}\label{LR to Post}
Let $(A,  \cdot , \alpha , \beta )$ be a regular BiHom-LR algebra. Then $(A, [\c,\c] ,\rhd , \alpha , \beta )$ is a BiHom-post-Lie algebra, where
\begin{equation}\label{CommutatorLR}
x\rhd y=-x\c y\ \ \textrm{and}\ \ [x,y]= x\c y- \a^{-1}\b(y)\c  \a\b^{-1}(x),\ \forall x,y\in A.
\end{equation}
\end{pro}
\begin{proof}
The BiHom-skew symmetry  is obvious. We will  check the BiHom-Jacobi identity. For all $x,y,z \in A,$ we have
\begin{align*}
\circlearrowleft_{x,y,z}\{\beta^2(x),\{\beta(y),\alpha(z)\}\}&=\beta^2(x)\c(\beta(y)\c\alpha(z))-(\alpha^{-1}\beta^2(y))\c\alpha\beta(z))\c\alpha\beta(x)\\
&-\beta^2(x)\c(\beta(z)\c\alpha(y))+(\alpha^{-1}\beta^2(z)\c\beta(y))\c\alpha\beta(x)\\
&+\beta^2(y)\c(\beta(z)\c\alpha(x))-(\alpha^{-1}\beta^2(z))\c\beta(x))\c\alpha\beta(y)\\
&-\beta^2(y)\c(\beta(x)\c\alpha(z))+(\alpha^{-1}\beta^2(x)\c\beta(z))\c\alpha\beta(y)\\
&+\beta^2(z)\c(\beta(x)\c\alpha(y))-(\alpha^{-1}\beta^2(x))\c\beta(y))\c\alpha\beta(z)\\
&-\beta^2(z)\c(\beta(y)\c\alpha(x))+(\alpha^{-1}\beta^2(y)\c\beta(x))\c\alpha\beta(z).
\end{align*}

By the identities \eqref{LRCondition1} and  \eqref{LRCondition2} of BiHom-LR  algebra, we have
$$\circlearrowleft_{x,y,z}\{\beta^2(x),\{\beta(y),\alpha(z)\}\}=0.$$
similarly, we have
\begin{align*}
   & \alpha \beta \left( x\right) \rhd \left[ y,z\right] -\left[ \beta
\left( y\right) , \alpha \left( x\right) \rhd z
\right] -\left[ \beta \left( x\right) \rhd y,\beta \left( z\right)
\right] \\
   =& -(\alpha \beta \left( x\right) \c \left[ y,z\right] -\left[ \beta
\left( y\right) , \alpha \left( x\right) \c z
\right] -\left[ \beta \left( x\right) \c y,\beta \left( z\right)
\right] )\\
=& -(\alpha \beta \left( x\right) \c \left( y\c z\right)-\alpha \beta \left( x\right) \c \left(\a^{-1}\b(z)\c\a\b^{-1}(y)\right) -\beta
\left( y\right) \c( \alpha \left( x\right) \c z
)\\&+( \b \left( x\right) \c \a^{-1}\b(z)
)\c\a
\left( y\right) - (\beta \left( x\right) \c y)\c\beta \left( z\right)+\a^{-1}\beta^2 \left( z\right)\c(\a \left( x\right) \c\a\b^{-1}( y))
)\\
=&0
\end{align*}
\end{proof}
\begin{df}\cite{Liu&Makhlouf&Menini&Panaite}\label{DefBiHomTriDend}
  A BiHom-tri-dendriform algebra is a 6-tuple $(A, \prec , \succ , \cdot , \alpha , \beta )$,
where $A$ is a linear space and $\prec , \succ , \cdot :A\otimes A\rightarrow A$ and
$\alpha , \beta :A\rightarrow A$ are linear maps
satisfying
\begin{eqnarray}
&&\alpha \circ \beta =\beta \circ \alpha , \label{BiHomtridend1} \\
&&\alpha (x\prec y)=\alpha (x)\prec \alpha (y), ~
\alpha (x\succ y)=\alpha (x)\succ \alpha (y), ~
\alpha (x\cdot y)=\alpha (x)\cdot \alpha (y), \label{BiHomtridend4} \\
&&\beta (x\prec y)=\beta (x)\prec \beta (y), ~
\beta (x\succ y)=\beta (x)\succ \beta (y), ~
\beta (x\cdot y)=\beta (x)\cdot \beta (y), \label{BiHomtridend7} \\
&&(x\prec y)\prec \beta (z)=\alpha (x)\prec (y\prec z+y\succ z+y\cdot z),  \label{BiHomtridend8} \\
&&(x \succ y)\prec \beta (z)=\alpha (x)\succ (y\prec z), \label{BiHomtridend9} \\
&&\alpha (x)\succ (y\succ z)=(x\prec y+x\succ y+x\cdot y)\succ \beta (z), \label{BiHomtridend10} \\
&&\alpha (x)\cdot (y\succ z)=(x\prec y)\cdot \beta (z), \label{BiHomtridend11} \\
&&\alpha (x)\succ (y\cdot z)=(x\succ y)\cdot \beta (z), \label{BiHomtridend12} \\
&&\alpha (x)\cdot (y\prec z)=(x\cdot y)\prec \beta (z), \label{BiHomtridend13} \\
&&\alpha (x)\cdot (y\cdot z)=(x\cdot y)\cdot \beta (z),  \label{BiHomtridend14}
\end{eqnarray}
for all $x, y, z\in A$.
We call $\alpha $ and $\beta $ (in this order) the structure maps
of $A$.

A morphism $f:(A, \prec , \succ , \cdot , \alpha , \beta )\rightarrow (A', \prec ', \succ ', \cdot ', \alpha ', \beta ')$ of
BiHom-tri-dendriform algebras is a linear map
$f:A\rightarrow A'$ satisfying $f(x\prec y)=f(x)\prec ' f(y)$, $f(x\succ y)=f(x)\succ ' f(y)$ and
$f(x\cdot y)=f(x)\cdot ' f(y)$,
for all $x, y\in A$,
as well as $f\circ \alpha =\alpha '\circ f$ and $f\circ \beta =\beta '\circ f$.
\end{df}
\begin{ex}
  Let $A =< e_1, e_2 > $ be a tow dimensional vector space. Given the multiplications
\begin{equation}\label{ExTriDen}
e_2 \prec e_2 = e_2 \succ e_2 = ae_1, \ e_2 \c e_2 = -ae_1,\end{equation}
and the  linear maps $ \alpha,\beta: A\to A$ defined by
\begin{align}\label{morphismExTriDen1}
   &\alpha(e_1) = e_1, \alpha(e_2) = e_1 +e_2,  \\
  &\beta(e_1) = e_1, \beta(e_2) = 2e_1 +e_2.
\end{align}
Then $(A, \prec , \succ , \cdot , \alpha , \beta )$ is a regular BiHom-tri-dendriform algebra.
\end{ex}
\begin{pro}
Let $(A,\prec,\succ,\c,\a,\b)$ be a regular BiHom-tri-dendriform algebra. Then $(A,[\c,\c],\rhd,\a,\b)$ is a BiHom-post-Lie algebra, where
\begin{align}
\label{tri to post 1}[x,y]=& x\c y- \a^{-1}\b(y)\c  \a\b^{-1}(x),\\
\label{tri to post 2} x \rhd y=& x \succ y -\a^{-1}\b(y)\prec \a\b^{-1}(x),
\end{align}
for any $x,y \in A$.
\end{pro}
\begin{proof}
For $x,y,z\in A$ we prove that%
\[
\alpha \beta \left( x\right) \rhd \left[ y,z\right] =\left[ \beta
\left( y\right) ,\left( \alpha \left( x\right) \rhd z\right) %
\right] +\left[ \beta \left( x\right) \rhd y,\beta \left( z\right) %
\right]
\]

We compute:%
\begin{eqnarray*}
\alpha \beta \left( x\right) \rhd \left[ y,z\right] &=&\alpha
\beta \left( x\right) \rhd \left( y\cdot z\right) -\alpha \beta
\left( x\right) \rhd \left( \alpha ^{-1}\beta \left( z\right)
\cdot \alpha \beta ^{-1}\left( y\right) \right) \\
&=&\alpha \beta \left( x\right) \succ \left( y\cdot z\right) -\left( \alpha
^{-1}\beta \left( y\right) \cdot \alpha ^{-1}\beta \left( z\right) \right)
\prec \alpha ^{2}\left( x\right) \\
\text{ \ \ }-\alpha \beta \left( x\right) &\succ &\left( \alpha ^{-1}\beta
\left( z\right) \cdot \alpha \beta ^{-1}\left( y\right) \right) +\left(
\alpha ^{-2}\beta \left( z\right) \cdot \left( y\right) \right) \prec \alpha
^{2}\left( x\right) .
\end{eqnarray*}

Then%
\begin{eqnarray*}
\left[ \beta \left( y\right) ,\left( \alpha \left( x\right) \rhd
z\right) \right] &=&\left[ \beta \left( y\right) ,\alpha \left( x\right)
\succ z\right] -\left[ \beta \left( y\right) ,\alpha ^{-1}\beta \left(
z\right) \prec \cdot \alpha ^{2}\beta ^{-1}\left( x\right) \right] \\
&=&\beta \left( y\right) \cdot \left( \alpha \left( x\right) \succ z\right)
-\beta \left( x\right) \succ \alpha ^{-1}\beta \left( z\right) \cdot \alpha
\left( y\right) \\
&&-\beta \left( y\right) \cdot \left( \alpha ^{-1}\beta \left( z\right)
\prec \cdot \alpha ^{2}\beta ^{-1}\left( x\right) \right) +\left( \alpha
^{-2}\beta ^{2}\left( z\right) \prec \alpha \left( x\right) \right) \cdot
\alpha \left( y\right)
\end{eqnarray*}

Similarly, we have
\begin{eqnarray*}
\left[ \beta \left( x\right) \rhd y,\beta \left( z\right) \right]
&=&\left[ \beta \left( x\right) \succ y,\beta \left( z\right) \right] -\left[
\alpha ^{-1}\beta \left( y\right) \prec \alpha \left( x\right) ,\beta \left(
z\right) \right] \\
&=&\beta \left( x\right) \succ y\cdot \beta \left( z\right) -\alpha
^{-1}\beta ^{2}\left( z\right) \cdot \left( \alpha \left( x\right) \succ
\alpha ^{-1}\beta \left( x\right) \right) \\
&&-\left( \alpha ^{-1}\beta \left( y\right) \prec \alpha \left( x\right)
\right) \cdot \beta \left( z\right) +\alpha ^{-1}\beta ^{2}\left( z\right)
\cdot \left( y\prec \alpha ^{2}\beta ^{-1}\left( x\right) \right)
\end{eqnarray*}

and the equality hold applying the relations \eqref{BiHomtridend12} and \eqref{BiHomtridend13} from  Definition \ref{DefBiHomTriDend}.
Similar, we can proof the second assertion by using \eqref{BiHomtridend7}-\eqref{BiHomtridend10}.
\end{proof}

It is easy to see that  Eq. \eqref{subadj lie alg}, Eq. \eqref{tri to post 1} and Eq. \eqref{tri to post 2} fit into the commutative diagram
$$
\xymatrix{
\text{BiHom-tri-dendriform alg.} \ar[rr]^{x\prec y + x\succ y + x\cdot y}
\ar[dd]^{x\circ y =x\succ y - \a^{-1}\b(y)\prec \a\b^{-1}(x) }_{[x,y]=x\cdot y -\a^{-1}\b(y)\cdot \a\b^{-1}(x)} &&
\mbox{BiHom-associative alg.} \ar[dd]^{x\star y - \a^{-1}\b(y)\star \a\b^{-1}(x)} \\
&&\\
\mbox{BiHom-post-Lie alg.} \ar[rr]^{x\circ y-\a^{-1}\b(y)\circ \a\b^{-1}(x)+[x,y]} && \text{BiHom-Lie alg.}
}
$$
When the operation $\cdot$ of the BiHom-tri-dendriform algebra and the bracket $[\c,\c,]$ of the BiHom-post-Lie algebra are both trivial, we obtain the following commutative diagram (See \cite{Liu&Makhlouf&Menini&Panaite,bihomass,rotabaxter on bihom ass} for more details).
$$
\xymatrix{ \mbox{BiHom-dendriform alg.} \ar[rr]^{x\prec y + x\succ y} \ar[dd]^{x\succ y - \a^{-1}\b(y)\prec \a\b^{-1}(x)} && \mbox{BiHom-ssociative alg.} \ar[dd]^{x\star y - \a^{-1}\b(y)\star \a\b^{-1}(x)}\\
&& \\
\mbox{BiHom-pre-Lie alg.} \ar[rr]^{x\circ y - \a^{-1}\b(y)\circ \a\b^{-1}(x)} && \mbox{BiHom-Lie alg.}
}
$$

\begin{thm}\label{BiHomLieToBiHomPostLie}
  Let $(A, [\c,\c], \a, \b)$ be a BiHom-Lie algebra and $(V, [\c,\c]_V, \r, \f, \p)$  an $A$-module $\mathbb{K}$-algebra.
The linear map $T: V\rightarrow A$  is  an $\mathcal{O}$-operator of weight $\lambda\in \mathbb{K}$ associated to
$(V, [\c,\c]_V, \r, \f, \p)$. Define two new bilinear operations $\{\c,\c\}, \vartriangleright: V\times V\rightarrow V$ as follows
\begin{eqnarray*}
\{u,v\}=\lambda[u,v]_V,~~u\vartriangleright v=\r(T(u))v,
\end{eqnarray*}
for any $u,v\in V$. Then $(V, \{\c,\c\}, \vartriangleright, \f, \p)$ is a BiHom-post-Lie algebra and $T$ is a morphism of BiHom-Lie algebras
from the associated BiHom-Lie algebra  of $(V, \{\c,\c\}, \vartriangleright, \f, \p)$ to $(A, [\c,\c], \a, \b)$. Furthermore, $T(V)$ is a BiHom-Lie subalgebra of $(A, [\c,\c], \a, \b)$ and there is an induced BiHom-post-Lie algebra structure on $T(V)$ given by
\begin{equation}\label{InducedBiHomPostLie}
  [T(u),T(v)]_{T(V)}=T(\{u,v\})\ \ , T(u)\vartriangleright_{T(V)}T(v)=T(u\vartriangleright v),\ \ \forall\ u,v\in V.
\end{equation}
\end{thm}
\begin{proof}
We use the last condition of representation of BiHom-Lie algebras on $\mathbb{K}$-algebra.
\begin{align*}
&\phi\psi(a)\rhd\{b,c\}-{\psi(a)\rhd b,\psi(c)}-\{\psi(b),\phi(a)\rhd c\}\\
&=\phi\psi(a)\rhd(\lambda[b,c]_{V})-\{\rho(T(\psi(a)))b,\psi(c)\}-\{\psi(b),\rho(T(\phi(a)))c\}\\
&=\rho(T(\phi\psi(a)))(\lambda[b,c]_{V})-\lambda[\rho(T(\psi(a)))b,\psi(c)]_{V}-\lambda[\psi(b),\rho(T(\phi(a)))c]_{V}\\
&=\lambda(\rho(\alpha\beta(T(a)))[b,c]_{V}-[\rho(\beta(T(a)))b,\psi(c)]-[\psi(b),\rho(\alpha(T(a)))c]_{V}\\
&=0.
\end{align*}
Using the condition \eqref{bihom-lie-rep-3} of Definition \ref{defi:bihom-lie representation}, we check
\begin{align*}
&\{\psi(a),\phi(b)\}\rhd\psi(c)-\phi\psi(a)\rhd(\phi(b)c)\\
&-(\psi(a)\rhd\phi(b))\rhd\psi(c)+\phi\psi(b)\rhd(\phi(a)\rhd c)+(\psi(b)\rhd\phi(a))\rhd\psi(c)\\
&=\rho(T(\rho(T(a)\phi(b))T(c)-\rho(T(\rho(T(\phi^{-1}\psi\phi(b)))\phi(a))\psi(a)\\
&+\rho(T(\{\psi(a),\phi(b)\}))\psi(c)-\phi\psi(a)(b \rhd c)+\psi\phi(b)\rhd(\phi(a) \rhd c)\\
&=\rho([T(\psi(a)),T(\phi(b)))])\psi(c)-\rho(T(\phi\psi(a)))\rho(T(a))c+\rho(T(\psi\phi(b)))\rho(T(\phi(a)))c\\
&=\rho(\beta(T(a)),T(\phi(b))])\psi(c)-\rho(\alpha\beta(T(a)))\rho(T(\phi(b)))c+\rho(\beta(T(\phi(b))))\rho(\alpha(T(a)))c\\
&=\rho([\beta(T(a)),T(\phi(b))])-\rho(\alpha\beta(T(a)))\rho(T(\phi(b)))c+\rho(\beta(T(\phi(b))))\rho(\alpha(T(a)))c\\
&=0.
\end{align*}
\end{proof}
An obvious consequence of Theorem \ref{BiHomLieToBiHomPostLie} is the following construction of a BiHom-post-Lie algebra in terms of Rota-Baxter operator of weight $\lambda$ on a BiHom-Lie algebra.

\begin{cor}
   Let $(A, [\c,\c], \a, \b)$ be a BiHom-Lie algebra. Then there exists  a compatible BiHom-post-Lie algebra structure on $A$ if and only if there exists an invertible $\mathcal{O}$-operator of weight $\lambda\in \mathbb{K}$ on $A$.
\end{cor}
\begin{proof}
Let $(A,[\c,\c],\rhd,\a,\b)$ be a BiHom-post-Lie algebra and $(A,[\c,\c],\a,\b)$ be the associated BiHom-Lie algebra.  Then the identity map $id: A \to A$ is an invertible $\mathcal{O}$-operator of weight $1$ on  $(A,[\c,\c],\a,\b)$  associated to $(A,[\c,\c],ad,\a,\b)$ .

Conversely, suppose that there exists an invertible $\mathcal{O}$-operator $T$ (of weight $\lambda$)  of $(A,[\c,\c],\a,\b)$  associated to  an $A$-module $\mathbb{K}$-algebra
 $(V, [\c,\c]_V, \r, \f, \p)$ .  Then, using Proposition \ref{BiHomLieToBiHomPostLie},  there is a BiHom-post-Lie algebra structure on $T(V)=A$ given by
\begin{equation*}
  \{T(u),T(v)\}=\lambda T([u,v]_V)\ \ , T(u)\vartriangleright T(v)=T(\rho(T(u))v),\ \ \forall\ u,v\in V.
\end{equation*}
If we st $x=T(u)$ and $y=T(v)$, then we get
\begin{equation*}
  \{x,y\}=\lambda T([T^{-1}(x),T^{-1}(y)]_V)\ \ , x\rhd y=T(\rho(x)T^{-1}(y)).
\end{equation*}
This is compatible BiHom-post-Lie algebra structure  on $(A,[\c,\c],\a,\b)$. Indeed,
\begin{align*}
& x\rhd y-\a^{-1}\b(y)\rhd \a\b^{-1}(x)+ \{x,y\} \\
&= T(\rho(x)T^{-1}(y)-\rho(\a^{-1}\b(y))T^{-1}\a\b^{-1}(x)+\lambda [T^{-1}(x),T^{-1}(y)]_V)\\
&=[TT^{-1}(x),TT^{-1}(y)]=[x,y].
\end{align*}
The proof is finished.

\end{proof}

\begin{cor}\label{PostLieByRotaBaxter}
  Let $(A, [\c,\c], \a, \b)$ be a BiHom-Lie algebra and the linear map $R: A\rightarrow A$  is  a  Rota-Baxter operator of weight $\lambda\in \mathbb{K}$. Then there exists a BiHom-post-Lie structure on $A$ given by
\begin{eqnarray*}
\{x,y\}=\lambda[x,y],~~x\vartriangleright y=[R(x),y],\ \forall \ x,y\in A.
\end{eqnarray*}
If in addition, $R$ is invertible, then there is a compatible BiHom-Post-Lie algebra structure on $A$ given by
\begin{equation*}
  [x,y]'=\lambda R([R^{-1}(x),R^{-1}(y)])\ \ , x\rhd y=R([x,R^{-1}(y)]),\ \forall x,y \in A.
\end{equation*}
\end{cor}
\begin{ex}
  Let $L=sl(2,\K)$  whose standard basis consists of
$$X=\left(
      \begin{array}{cc}
        0 & 1 \\
        0 & 0 \\
      \end{array}
    \right),
Y=\left(
    \begin{array}{cc}
      0 & 0 \\
      1 & 0 \\
    \end{array}
  \right),
  H=\left(
      \begin{array}{cc}
        1 & 0 \\
        0 & -1 \\
      \end{array}
    \right).
$$
Then $[H,X]=2X,\quad [H,Y]=-2Y,\quad [X,Y]=H.$  Define two linear maps
$\alpha, \beta: L\rightarrow L$  by
$$\alpha(X)=\lambda^2 X, \quad \alpha(Y)=\frac{1}{\lambda^2}Y,\quad \alpha(H)=H,$$
$$\beta(X)=\gamma^2X, \quad \beta(Y)=\frac{1}{\gamma^2}Y,\quad \beta(H)=H,$$
where $\lambda,\gamma$ are parameters in $\K$.
Obviously we check that $\alpha$ and $\beta$ are two morphisms of the Lie algebra $(L,[\c,\c])$.
 Consider  the
 linear map $ [\c,\c]_{\a,\b}\colon L\otimes L\rightarrow L$
 \begin{gather*}
 [a,b]_{\a,\b} = [ \alpha (a) ,\beta (b) ] ,\qquad \text{for all} \ \ a,b\in L
\end{gather*} defined in the basis $X,\ Y, \ H$ by
$$[H,X]_{\a,\b}=2\gamma^2 X,\ \  [H,Y]_{\a,\b} =-\frac{2}{\gamma^2}Y,\ \  [X,Y]_{\a,\b}=\frac{\lambda^2}{\gamma^2}H. $$
Then $L_{( \alpha ,\beta ) }:=(L, [\c,\c]_{\a,\b}, \alpha ,
\beta )$ is a BiHom-Lie algebra.

Now, define the linear map $R:L\to L$ by
\begin{equation}\label{RotaBaxterSl2}
  R(X)=0,\ \ R(Y)=4Y\ \ \textrm{and}\ \ R(H)=2H.
\end{equation}
Then $R$ is a Rota-Baxter operator of weight $-4$ on $L_{( \alpha ,\beta ) }$. Using Corollary \ref{PostLieByRotaBaxter}, we can construct a BiHom-Post-Lie algebra on $L_{( \alpha ,\beta ) }$ given by
$$\{H,X\}=-8\gamma^2 X,\ \ \{H,Y\} =\frac{8}{\gamma^2}Y,\ \  \{X,Y\}=-4\frac{\lambda^2}{\gamma^2}H $$
and
$$X\rhd Y=X\rhd H=0,\ \ Y\rhd X=-4\frac{\gamma^2}{\lambda^2}H,\ \ Y\rhd H=-\frac{8}{\lambda^2}Y,\ \ H\rhd X=4\gamma^2X,\ \ H\rhd Y=-\frac{4}{\gamma^2}Y.$$
\end{ex}
\begin{ex}
  Let $(A, [\c,\c], \a, \b)$ be a BiHom-Lie algebra such that $A=A_1\oplus A_2$, where $A_1$ and $ A_2$  are two BiHom-Lie subalgebras, and the linear map $R: A\rightarrow A$  given by
$$R(x_1+x_2)=-\lambda x_2, \forall\ x_1\in A_1, \ x_2\in A_2.$$
It is easy to check that  $R$ is a   Rota-Baxter operator of weight $\lambda\in \mathbb{K}$ on $A$. Then $(A,\vartriangleright, \{\c,\c\}, \a, \b)$ is a BiHom-post-Lie algebra, where
\begin{eqnarray*}
\{x,y\}=\lambda[x,y],~~x\vartriangleright y=[R(x),y],\ \forall \ x,y\in A.
\end{eqnarray*}
\end{ex}
\section{Representation theory}
In this section, we introduce the notion of representations of a BiHom-post-Lie algebra $(A,[\c,\c],\rhd,\a,\b)$ on a vector space $V$.
 \begin{df}\label{rep bihom postlie}
 A representation of a   BiHom-post-Lie algebra $(A,[\c,\c],\rhd,\a,\b)$ on a vector space $V$ is a tuple $(V,\rho,\mu,\nu,\f,\p)$, such that $(V,\rho,\f,\p)$ is a representation of
$(A,[\c,\c],\a,\b)$  and $\mu,\nu: A \to gl(V)$ are linear maps satisfying

\begin{align}
&\mu(\alpha(x)) \f=\f \mu(x),\ \  \mu(\b(x)) \p=\p \mu(x),   \\
&\nu(\alpha(x) \f=\f \nu(x),\ \  \nu(\b(x)) \p=\p \nu(x), \\
& \nu([x,y])  \f \p=\rho(\b(x)) \nu(y)  \f-\rho(\b(y)) \nu(x)  \p , \label{rep bihom pos lie 1}\\
& \rho(\b(x)\rhd y) \p=\mu(\a\b(x)) \rho(y)-\rho(\b(y))  \mu(\a(x))  , \label{rep bihom pos lie 2}\\
&\mu([\b(x),\a(y)])  \p=\mu(\a\b(x))\mu(\a(y))-\mu(\b(x)\rhd \a(y))  \p \nonumber\\
& \hspace{0.5 cm}-\mu(\a\b(y))\mu(\a(x))+\mu(\b(y)\rhd \a(x))  \p   , \label{rep bihom pos lie 3}\\
& \nu(\b(y))\rho(\b(x))  \f=\mu(\a\b(x))\nu(y) \f-\nu(\b(y))\mu(\b(x))\f \nonumber\\
& \hspace{0.5 cm}-\nu(\a(x)\rhd y)\f\p +\nu(\b(y))\nu(\a(x)) \p ,  \label{rep bihom pos lie 4}
\end{align}
for any $x,y \in A$.
 \end{df}

Let  $(A,[\c,\c],\rhd,\a,\b)$ be a BiHom-post-Lie algebra and  $(V,\rho,\mu,\nu,\f,\p)$  a representation of   $(A,[\c,\c],\rhd,\a,\b)$.  By identity \eqref{rep bihom pos lie 3}, we deduce that $(V,\mu,\f,\p)$ is a representation of the sub-adjacent BiHom-Lie algebra $(A,\{\c,\c\},\a,\b)$ of   $(A,[\c,\c],\rhd,\a,\b)$.  In addition, it is obvious that $(A,ad,L_{\rhd},R_{\rhd},\a,\b)$ is a representation of  $(A,[\c,\c],\rhd,\a,\b)$  which is called the adjoint representation.

\begin{thm}\label{semi-directproduct post-lie alg}
A tuple  $(V,\rho,\mu,\nu,\f,\p)$ is a representation of a BiHom-post-Lie algebra   \\ $(A,[\c,\c],\rhd,\a,\b)$ if and only if $(A\oplus V,[\c,\c]_{\rho},\rhd_{\mu,\nu},\a+\f,\b+\p)$ is a BiHom-post-Lie algebra, where for any $x,y \in A$ and $u,v \in V$
\begin{align*}
&(\alpha+\f)(x+u)=\alpha(x)+\f(u),\\
&(\b+\p)(x+u)=\b(x)+\p(u), \\
&[x+u,y+v]_{\rho}=[x,y]+\rho(x)v-\rho(\a^{-1}\b(y))\f\p^{-1}(u),\\
&(x+u)\rhd_{\mu,\nu}(y+v)=x\rhd y+\mu(x)v+\nu(y)u.
\end{align*}
\end{thm}

\begin{proof}
  By the conditions \eqref{rep bihom pos lie 3} and \eqref{rep bihom pos lie 4} in Definition  \ref{rep bihom postlie} and the identity \eqref{bihom-post-lie condition2}, we have
\begin{align*}
&[(\beta+\psi)(x+a),(\alpha+\phi)(y+b)]_{\rho}\rhd_{\mu,\nu}(\beta+\psi)(z+c)\\
&-as_{\alpha+\phi,\beta+\psi}((\beta+\psi)(x+a),(\alpha+\phi)(y+b),z+c)\\
&+as_{\alpha+\phi,\beta+\psi}((\beta+\psi)(y+b),(\alpha+\phi)(x+a),z+c)\\
&=[\beta(x),\alpha(y)]\rhd\beta(z)+\mu([\beta(x),\alpha(y)])\psi(c)+\nu(\beta(z))(\rho(\beta(x))\phi(b)-\rho(\beta(y))\phi(a))\\
&-\alpha\beta(x)\rhd(\alpha(y)\rhd z)-\mu(\alpha\beta(x))\mu(\alpha(y))c-\mu(\alpha\beta(x))\nu(z)\phi(b)-\nu(\alpha(y)\rhd z)\phi\psi(a)\\
&+(\beta(x)\rhd\alpha(y))\rhd\beta(z)+\mu(\beta(x)\rhd\alpha(y))\psi(c)+\nu(\beta(z))\mu(\beta(x))\phi(b)+\nu(\beta(z))\nu(\alpha(y))\psi(a)\\
&+\alpha\beta(y)\rhd(\alpha(x)\rhd z)+\mu(\alpha\beta(y))\mu(\alpha(x))c+\mu(\alpha\beta(y))\nu(z)\phi(a)+\nu(\alpha(x)\rhd z)\phi\psi(b)\\
&-(\beta(y)\rhd\alpha(x))\rhd\beta(z)-\mu(\beta(y)\rhd\alpha(x))\psi(c)-\nu(\beta(z))\mu(\beta(y))\phi(a)-\nu(\beta(z))\nu(\alpha(x))\psi(b)\\
&=0.
\end{align*}
Similarly we can check the identity \eqref{bihom-post-lie condition1} for the algebra $A\bigoplus V$ using  the axioms \eqref{rep bihom pos lie 1}--\eqref{rep bihom pos lie 4} and \eqref{bihom-post-lie condition1}.
\end{proof}

Let  $(A,[\c,\c],\rhd)$ be a post-Lie algebra and $\a,\b$ be two commuting post-Lie algebra morphisms on $A$.
Consider a representation  $(V,\rho,\mu,\nu)$  $(A,[\c,\c],\rhd)$ and $\f,\p: V \to V$ be two commuting linear maps obeying two the following conditions
\begin{align*}
&  \rho(\alpha(x))\circ \f=\f\circ \rho(x),\ \ \rho(\b(x))\circ \p=\p \circ \rho(x), \\
&  \mu(\alpha(x))\circ \f=\f\circ \mu(x),\ \ \mu(\b(x))\circ \p=\p \circ \mu(x), \\
&  \nu(\alpha(x))\circ \f=\f\circ \nu(x),\ \ \nu(\b(x))\circ \p=\p \circ \nu(x).\\
\end{align*}
  Define the three  linear maps   $\widetilde{\rho}, \widetilde{\mu}, \widetilde{\nu}: A\longrightarrow gl(V)$ by \\$\widetilde{\rho}(x)(v)=\rho(\alpha(x))( \p(v))$,   $\widetilde{\mu}(x)(v)=\mu(\alpha(x))( \p(v))$ and
   $\widetilde{\nu}(x)(v)=\nu(\alpha(x))( \p(v))$.
  \begin{pro}\label{twistmodule}
The tuple$(V,\widetilde{\rho},\widetilde{\mu},\widetilde{\nu},\f,\p)$ is a representation of $(A,[\cdot,\cdot]_{\a,\b},\rhd_{\a,\b},\alpha,\beta)$.
  \end{pro}

\begin{proof}
straightforward.
\end{proof}

Let   $(V,\rho,\mu,\nu,\f,\p)$ be a representation of a BiHom-post-Lie algebra   $(A,[\c,\c],\rhd,\a,\b)$. Define the linear map
$\pi:A  \to gl(V)$ by
\begin{align}
\pi(x):=\rho(x)+\mu(u)-\nu(\a\b^{-1}(x))\f^{-1}\p,\ \forall x \in A.
\end{align}
\begin{pro}
With the above notations, $(V,\pi,\f,\p)$ is  a representation of the sub-adjacent BiHom-Lie algebra $(A,\{\c,\c\},\a,\b)$ of   $(A,[\c,\c],\rhd,\a,\b)$.
\end{pro}

\begin{proof}
By Theorem \ref{semi-directproduct post-lie alg}, we have the semi-direct product BiHom-post-Lie algebra $(A\oplus V,[\c,\c]_{\rho},\rhd_{\mu,\nu},\a+\f,\b+\p)$. Consider its sub-adjacent Lie  algebra structure $[\c,\c]^C$, we have
{\small
\begin{align*}
& [x+u,y+v]^C=(x+u)\rhd_{\mu,\nu} (y+v)- (\a^{-1}\b(y)+\f^{-1}\p(v)) \rhd_{\mu,\nu} (\a\b^{-1}(x)+\f\p^{-1}(u)) +[x+u,y+v]_{\rho} \\
&= x\rhd y+\mu(x)v+\nu(y)u-
\a^{-1}\b(y) \rhd \a\b^{-1}(x)-\mu(\a^{-1}\b(y))\f\p^{-1}(u)-\nu(\a\b^{-1}(x))\f^{-1}\p(v) \\
& + [x,y]+\rho(x)v-\rho(\a^{-1}\b(y))\f\p^{-1}(u) \\
&=\{x,y\} +\pi(x)v-\pi(\a^{-1}\b(y))\f\p^{-1}(u).
\end{align*}
}
Then $(V,\pi,\f,\p)$ is  a representation of the sub-adjacent BiHom-Lie algebra $(A,\{\c,\c\},\a,\b)$.

\end{proof}


\end{document}